\documentclass[12pt]{amsart}
\usepackage{amsmath,amssymb,amsfonts, amscd}
\usepackage{amsbsy}
\usepackage{amsthm}
\usepackage{amstext}
\usepackage{amsopn}
\usepackage{hyperref}
\usepackage{mathrsfs}
\usepackage{graphicx}
\usepackage{latexsym}
\usepackage[T1]{fontenc}
\usepackage{color}

\newcommand{\Hmm}[1]{\leavevmode{\marginpar{\tiny%
$\hbox to 0mm{\hspace*{-0.5mm}$\leftarrow$\hss}%
\vcenter{\vrule depth 0.1mm height 0.1mm width \the\marginparwidth}%
\hbox to 0mm{\hss$\rightarrow$\hspace*{-0.5mm}}$\\\relax\raggedright #1}}}

\newtheorem{thm}{Theorem}[section]

\newtheorem{lemma}[thm]{Lemma}

\theoremstyle{definition}

\newtheorem{eg}[thm]{Example}
\newtheorem{rem}[thm]{Remark}

\newcommand{\R}{{\mathbb R}}

\newcommand{\N}{{\mathbb N}}

\newcommand{\al}{{\alpha}}

\newcommand{\eps}{{\varepsilon}}
\newcommand{\gm}{{\gamma}}

\newcommand{\si}{{\sigma}}
\newcommand{\lm}{{\lambda}}
\newcommand{\ph}{{\varphi}}

\begin{document}
\title[Rapidly branching graphs]{Geometry and spectrum of rapidly branching graphs}

\author[M. Keller]{Matthias Keller}
\author[F. M\"unch]{Florentin  M\"unch}
\address{M. Keller, F. M\"unch,
Mathematisches Institut \\Friedrich Schiller Universit{\"a}t Jena \\07743 Jena, Germany }
\email{florentin.muench@uni-jena.de}
\email{m.keller@uni-jena.de}
\author[F.~Pogorzelski]{Felix Pogorzelski}
\address{F.~Pogorzelski,  Department of Mathematics, Technion - Israel Institute of Technology, 32000 Haifa, Israel}
\email{felixp@technion.ac.il}
% tx.technion.ac.il}

%\begin{abstract} \noindent survey planar graphs. curvature. focus on spectrum . \end{abstract}

\date{\today}
%\maketitle
\maketitle
%\tableofcontents

% \frenchspacing
%%%%%%%%%%%%%%%%%%%%%%%%%%%%%%%%%%%%%%%%%%%%%%%%%%%%%%%%%%%%
% ABSTRACT
%%%%%%%%%%%%%%%%%%%%%%%%%%%%%%%%%%%%%%%%%%%%%%%%%%%%%%%%%%%%
\begin{abstract}
We study graphs whose vertex degree tends and which are, therefore, called rapidly branching. We prove spectral estimates, discreteness of spectrum, first order eigenvalue and Weyl asymptotics solely in  terms of the vertex degree growth. The underlying techniques are estimates on the isoperimetric constant. Furthermore, we give lower volume growth bounds and we provide a new criterion for stochastic incompleteness.
\end{abstract}

\section{Introduction}
The  spectral theory of Laplacians on graph is a vibrant topic of study. A specific focus lies on geometric
criteria for spectral bounds and  discreteness of spectrum, see e.g.
\cite{Do,DK,BGK1,Fu2,Gol2,K1,KL2,KPP,Mo91,Woj2,Zuk}. A necessary condition for
discreteness of spectrum is that the vertex degree tends to infinity, a condition which  is called rapidly branching, \cite{Fu2}. This
condition is by no means sufficient. Throughout the years various
geometric criteria were given to ensure that rapid branching
implies purely discrete spectrum, see e.g. \cite{BGK1,Fu2,K1,KPP,Woj2}. \\
The novelty of this work is to provide one single and concise criterion
on the growth of the vertex degree which by itself implies discreteness of spectrum.
Specifically, we consider graphs
whose vertex degrees grow proportionally to the numbers of vertices
with smaller degree. In this sense we control the acceleration of the vertex degree
growth.
% new approach for showing discreteness of
%spectrum for graphs with a very high acceleration of the growth of the vertex
%degrees.
%In this
%work we give a growth criterion of the vertex degree which by itself
%implies discreteness of spectrum.
%
We discover that for these graphs,
there is no need to impose a priori conditions on the underlying
geometry such as planarity or curvature assumptions.
Furthermore, we derive valuable geometric implications
such as bounds for the isoperimetric
constant at infinity and volume growth bounds.

Moreover, we study stochastic completeness which is a topic that
has also been investigated intensively for graphs with unbounded degree
throughout the recent years, cf.\@ e.g.\@
\cite{Fol2,GHM,Hu11,Hu2,KL1,KL2,KLW,Woj1,Woj2,Woj3}. We show that, for a certain
acceleration of the vertex degree growth, the graphs are stochastically
incomplete.

\subsection{Set-up and definitions}
Let $G=(V,E)$ be an infinite, simple, locally finite,
connected graph. The {vertex degree} $\deg:V\to\N$ assigns
to each vertex $v$ the number of edges emanating from $v$.
Furthermore, denote $\eta:\R\to\N\cup\{\infty\}$
\begin{align*}
    \eta(k)=\#\{v\in V\mid \deg(v)\leq k\}.
\end{align*}
We are interested in graphs for which $\eta$ grows very slowly. That is
for each bound $k$ there are only few vertices with degree smaller than $k$.  This is measured by the following constants
\begin{align*}
    r&=\sup_{k\in\N}\frac{\eta(k)}{k}\quad\mbox{and}\quad
    r_{\infty}=\limsup_{k\to\infty}\frac{\eta(k)}{k}.
\end{align*}
Clearly, $0\leq r_{\infty}\leq r\leq \infty$. We can think of $1/r$ or $1/r_{\infty}$ as a type of acceleration rate.
One can also see $\eta$ as the spectral counting function of the multiplication operator with the vertex degree function.

Let
\begin{align*}
    d=\min_{v\in V}\deg(v)\quad\mbox{and}\quad
    d_{\infty}=\sup_{K\subseteq V\mbox{\scriptsize{ finite}}}\min_{v\in V\setminus K}\deg(v).
\end{align*}
In \cite{Fu2,K1} graphs with $d_{\infty}=\infty$ are referred to as
\emph{rapidly branching}. In this case we enumerate the vertices
$V=\{v_{k}\}_{k\ge0}$ such that
$$\deg(v_{k})\leq\deg(v_{k+1})$$
and define
\begin{align*}
    d_{k}=\deg(v_{k})
\end{align*}
for  $k\ge0$. Obviously, $d\leq d_{k}\nearrow d_{\infty}$ as  $k\to\infty$.

The main focus of this paper is the spectral theory of the
Laplacian
\begin{align*}
    \Delta\ph(v)=\sum_{w\sim v}(\ph(v)-\ph(w))
\end{align*}
which is a positive selfadjoint operator on the Hilbert space
$\ell^{2}(V)$ of real valued square summable functions with domain
\begin{align*}
    D(\Delta)=\{\ph\in \ell^{2}(V)\mid \Delta\ph\in\ell^{2}(V)\},
\end{align*}
see \cite[Section~1.3]{Woj1}. We denote the bottom of  the spectrum of $\Delta$ by $\lm_{0}=\lm_{0}(\Delta)$.
We say that the spectrum  of $\Delta$
is \emph{purely discrete} if it consists only of isolated
eigenvalues of finite multiplicity. It is easy to see that a
necessary condition for purely discrete spectrum is
$d_{\infty}=\infty$, see e.g. \cite[Proposition~5]{K1}. An example to see that this condition is not sufficient can be found in \cite[Theorem~6.1]{K1}.

In the case of purely discrete eigenvalues we enumerate the
eigenvalues of $\Delta $ in increasing order and counted with
multiplicity by $\lm_{k}$, $k\ge0$. Moreover, we denote the
eigenvalue counting function of $\Delta$ by
\begin{align*}
    N(\lm)=\sup\{k\ge0\mid \lm_{k}\le
    \lm\},
\end{align*}
where $\si(\Delta)$ is the spectrum of $\Delta$.

\subsection{Results}
First we  present the results on  the spectral
theory of $\Delta$. We denote
\begin{align*}
    \gm(s)=\sqrt{1-\frac{(1-s)^{4}}{(1+s^{2})^{2}}}
\end{align*}
for $s\in[0,1]$ and notice that $\gm(s)\in[0,1]$ while $\gm(0)=0$ and $\gm(1)=1$.

The first result deals with the spectral gap.

\begin{thm}\label{t:bottom} If $d_{\infty}=\infty$ and $r_{\infty}<1$, then
$\lm_{0}>0$. If even $r<1$,  then
\begin{align*}
d(1-\gm(r))\le     \lm_{0}.
\end{align*}
\end{thm}

Secondly, we give a criterion for discreteness of the spectrum.
Furthermore, we provide estimates on the corresponding Weyl and
eigenvalue asymptotics.

\begin{thm}\label{t:discrete} If $d_{\infty}=\infty$ and $r_{\infty}<1$,   then $\Delta$
has purely discrete spectrum and we have
\begin{align*}
1-\gm(r)\leq\liminf_{\lm\to\infty}\frac{N(\lm)}{\eta(\lm)}\le
\limsup_{\lm\to\infty}\frac{N(\lm)}{\eta(\lm)}\leq1+\gm(r)
\end{align*}
and
\begin{align*}
1-\gm(r_{\infty})\leq\liminf_{k\to\infty}\frac{\lm_{k}}{d_{k}}\le
\limsup_{k\to\infty}\frac{\lm_{k}}{d_{k}}\leq1+\gm(r_{\infty}).
\end{align*}
In particular, if $r_{\infty}=0$, then $\lm_{k}/d_{k}\to1$ as
$k\to\infty$.
\end{thm}

The results of the theorems above are based on the following
estimate of the isoperimetric constant $\al$ and the isoperimetric
constant at infinity $\al_{\infty}$ defined as
\begin{align*}
\al=\inf_{U\subseteq V \mbox{\scriptsize{ finite}}} \frac{\#\partial
U}{\mathrm{vol}(U)} \quad\mbox{and}\quad
\al_{\infty}=\sup_{K\subseteq V\mbox{\scriptsize{
finite}}}\inf_{U\subseteq V\setminus K \mbox{\scriptsize{ finite}}}
\frac{\#\partial U}{\mathrm{vol}(U)},
\end{align*}
where $\partial U= \{(v,w)\in U\times (V\setminus U)\mid v\sim w\}$
and $\mathrm{vol}(U)=\sum_{v\in U}\deg(v)$.

\begin{thm}\label{t:isoperimetric} Assume $d_{\infty}=\infty$. Then,
\begin{align*}
\al&\ge1-\frac{2r}{1+r^{2}} \quad\mbox{if}\quad
r \le 1  \quad\mbox{and} \\ \al_{\infty}&\ge1-\frac{2r_{\infty}}{1+r_{\infty}^{2}} \quad\mbox{if}\quad
r_{\infty} \le 1 .
\end{align*}
\end{thm}
The proofs of these theorems are found in
Section~\ref{s:isoperimetric}. Let us mention that the results above
are sharp in the sense that for given $s\in(0,1)$ there are graphs
$G_{s}$ such that $s=r_{\infty}$ and such that the inequality for
$\al_{\infty}$ is an equality.

\begin{rem} Next to the operator $\Delta$, one often  considers the
normalized Laplacian $\Delta_{n}$ that is the bounded, positive,
selfadjoint operator  acting as
\begin{align*}
    \Delta_{n}\ph(v)=\frac{1}{\deg(v)}\sum_{w\sim v}(\ph(v)-\ph(w))
\end{align*}
on $\ell^{2}(V,\deg)$. For example this operator is studied in
\cite{DK,Fu2}. By the considerations of \cite{Fu2} and
Theorem~\ref{t:isoperimetric} above,  we obtain the spectral
estimates $\lm_{0}(\Delta_{n})\ge\gm(r)$ and
$\lm_{0}^{\mathrm{ess}}(\Delta_{n})\ge 1-\gm(r_{\infty})$, where
$\lm_{0}(\Delta_{n})$ and $\lm_{0}^{\mathrm{ess}}(\Delta_{n})$ is
the bottom of the spectrum and the bottom of the essential spectrum
of $\Delta_{n}$. Furthermore, if $r_{\infty}=0$, then the essential
spectrum of $\Delta_{n}$ is equal to $\{1\}$.
\end{rem}

Next to isoperimetric estimates, we show a lower exponential volume
growth bound for graphs with $r_{\infty}<1$. For a given vertex
$v\in V$, we denote by $B_{n}$ the vertices which can be connected to
$v$ by a path of less or equal to $n$ edges.

\begin{thm}\label{t:volume}
Assume $d_{\infty}=\infty$ and
$0<r_{\infty} < 1$.
Let $a$ be the largest real root of the polynomial
\[
p(z) = z^3 - r_\infty^{-1} z^2 - 1.
\]
Then,
$$\liminf_{n\to\infty}\frac{1}{n}\log\mathrm{vol}(B_{n})\ge 2\, \log a.$$
In particular, $a \ge r_\infty^{-1}$.
\end{thm}
This result is proven in Section~\ref{s:volume}. Of course, there
are many examples of graphs with $r_{\infty}>0$ with
superexponential volume growth such as trees. However, in
Section~\ref{s:volume} we present also an example showing that the
result is sharp.

Finally, we turn to a property of graphs called stochastic
completeness. In the discrete setting the investigation of this topic goes
back to Feller \cite{Fe1,Fe2} and Reuter \cite{Reu}. Recent interest
in this topic was sparked by the Ph.D. thesis of  Wojciechowski
\cite{Woj1}, see also \cite{GHM,Hu1,Hu2,Fol3,KLW,Woj2,Woj3} and
references therein.

A graph is said to be \emph{stochastically complete} if
$$e^{-t\Delta}1=1,\qquad t\ge0,$$
where $e^{-t\Delta}$ is the $\ell^{2}$ semigroup of $\Delta$ which
is extended to the bounded functions via monotone limits and $1$
denotes the function that is constantly one on $V$. For details, see
\cite{Woj1}. Note that by general theory one always has
$e^{-t\Delta}\leq 1$ and a graph is  said to be \emph{stochastically
incomplete} if
\begin{align*}
e^{-t\Delta}1<1,\mbox{ for some (all) }t>0.
\end{align*}

The importance of this property stems from the fact that stochastic completeness is
equivalent to uniqueness of bounded solutions to the heat equation,
see \cite[Theorem 3.1.3]{Woj1}. Intuitively, stochastic incompleteness can be understood that heat vanishes from the graph in finite time.

Here, we give a new criterion for stochastic incompleteness.

\begin{thm}\label{t:sc} Let $d_{\infty}=\infty$. If
$r_{\infty}<e^{-1}$, then the graph is stochastically incomplete.
\end{thm}
The constant $e$ in the theorem denotes the Euler number. The proof
is given in Section~\ref{s:sc}.

The authors conjecture that the statement of the theorem remains
true for $r_{\infty}<1$. However, the idea of proof given here does
not extend to this situation.

\section{Isoperimetric constants}\label{s:isoperimetric}
In this section we show Theorem~\ref{t:isoperimetric} from which we
deduce Theorem~\ref{t:bottom} and Theorem~\ref{t:discrete}.

We start with a basic but important  fact which will be used
successively throughout the paper. Recall that we enumerated the vertices $V=\{v_{k}\}_{k\ge0}$ with respect to increasing vertex degree.

\begin{lemma}\label{l:basic} For $s>r_{\infty}$ there is $n\ge0$ such that for  $k\ge n$ we have
\begin{align*}
    \frac{(k+1)}{s}\leq\deg(v_{k}).
\end{align*}
\end{lemma}
\begin{proof}
By the inclusion $\{v_{0},\ldots,v_{k}\}\subseteq\{v\in V\mid
\deg(v)\leq\deg(v_{k})\}$ we have
\begin{align*}
    (k+1)\leq \eta(\deg(v_{k})).
\end{align*}
Let now $n$ be such that $\eta(\deg(v_{k}))\leq s \deg(v_{k})$ for
$k\ge n$. Combining this with the inequality above yields the
statement.
\end{proof}

Now, we give the proof of Theorem~\ref{t:isoperimetric}. We recall
the notation  $a\wedge b= \min\{a,b\}$, $a\vee b= \max\{a,b\}$ and
$a_{+}=a\vee 0$ for numbers $a,b\in\R$.

\begin{proof}[Proof of Theorem~\ref{t:isoperimetric}]
Let $U\subseteq V$ and let $N\ge0$ be such that $\#U=N+1$. Since
every vertex in $U$ can be connected to at most $N$ vertices within
$U$, we estimate
\begin{align*}
    \#\partial U\geq\sum_{v\in U}(\deg(v)-N)_{+}= -N(N+1)+\sum_{v\in U}(\deg(v)\vee N),
\end{align*}
and
\begin{align*}    \mathrm{vol}(U)\leq\sum_{v\in U}(\deg(v)\vee N).
\end{align*}
Hence,
\begin{align*}
    \frac{\#\partial U}{\mathrm{vol}(U)}\ge 1-\frac{N(N+1)}{\sum_{v\in U}(\deg(v)\vee N)}.
\end{align*}
For $s\in(r_{\infty},1]$, let $n\ge1$ be chosen according to
Lemma~\ref{l:basic}.
For $K=\{v_{0},\ldots, v_{n-1}\}$ and $U
\subseteq V\setminus K$ with $\#U=N+1$, we conclude by
Lemma~\ref{l:basic}
\begin{align*}
\sum_{v\in U}(\deg(v)\vee N) &\ge\sum_{k=n}^{n+N}(\deg(v_{k})\vee N)
\ge \sum_{k=n}^{n+N}\left(\frac{(k+1)}{s}\vee N\right)\\
&\ge \sum_{k=1}^{N+1}\left(\frac{k}{s}\vee N\right)
\ge\frac{1}{s}\int_{0}^{N+1}(k\vee sN)dk\\
&=\frac{1}{s}\Big((sN)^{2}+\frac{1}{2}\left((N+1)^{2}-(sN)^{2}\right)\Big)\\
&\ge \frac{s^{2}+1}{2s}N(N+1).
\end{align*}
Plugging this in the inequality for  $\#\partial
U/{\mathrm{vol}(U)}$ above, we conclude the statement for
$\al_{\infty}$. The corresponding statement for $\al$ follows
analogously by letting $s=r$ and $n=0$.
\end{proof}

The proofs of Theorem~\ref{t:bottom} and Theorem~\ref{t:discrete}
are based on Cheeger estimates for which there is a huge body of
literature, see e.g.\@ \cite{Do,DK,Fu2,Mo88,Mo91,KL2}.
We denote the functions of finite support on $V$ by $C_{c}(V)$ and
denote the scalar product of $\ell^{2}(V)$ by
$\langle\cdot,\cdot\rangle$.

\begin{proof}[Proof of  Theorem~\ref{t:bottom} and Theorem~\ref{t:discrete}]
The following inequality  can be directly extracted from
\cite[Proof of Proposition~15]{KL2}
\begin{align*}
    (1-\sqrt{1-\al^{2}})\langle\deg \ph,\ph\rangle\leq
\langle\Delta \ph,\ph\rangle
    \leq    (1+\sqrt{1-\al^{2}})\langle\deg \ph,\ph\rangle,
\end{align*}
for all $\ph\in C_{c}(V)$. This yields the bound in
Theorem~\ref{t:bottom}. Furthermore, from \cite[Theorem~5.3]{BGK1}
it follows that for all $\eps>0$ there is $C_{\eps}>0$ such that for
all normalized functions $\ph$ with finite support
\begin{align*}
  (1-\eps)(1-\sqrt{1-\al_{\infty}^{2}})\langle\deg \ph&,\ph\rangle -C_{\eps}\leq
\langle\Delta \ph,\ph\rangle\\
    &\leq   (1+\eps) (1+\sqrt{1-\al_{\infty}^{2}})\langle\deg \ph,\ph\rangle+C_{\eps}.
\end{align*}
By the Min-Max-Principle \cite[Chapter~XIII.1]{RS} (confer
\cite[Theorem~A.2]{BGK1} or \cite{Gol2} for the details of the
application) we deduce the statement about discreteness of spectrum
as well as the Weyl and eigenvalue asymptotics. This proves
Theorem~\ref{t:discrete}. Finally, as the spectrum is purely
discrete if $r_{\infty}<1$, we deduce that $\lm_{0}>0$. To see this,
assume for a moment that
$\lm_{0}=0$ is an eigenvalue. Noting
that
\[
0 = \langle \Delta \varphi, \varphi \rangle =
\frac{1}{2}\sum_{x\sim y} \big( \varphi(x) - \varphi(y) \big)^2
\]
for the underlying eigenfunction $\varphi$ shows that $\varphi$
must be a constant function. However, the only constant function
in $\ell^2(V)$ is the zero function. Hence, $\varphi$ is not an
eigenfunction and $0$ cannot be an eigenvalue.
%no non-trivial harmonic function is in $\ell^{2}(V)$ by a maximum principle, \cite[Theorem~4.4 and Lemma 2.1]{K4}.
This finishes the proof of Theorem~\ref{t:bottom}.
\end{proof}

Let us turn to an example which shows that the bound in
Theorem~\ref{t:isoperimetric} above is sharp.

\begin{eg} We construct a graph for given $s\in(0,1)$ as follows. Let
$K_{n}=(V_{n},E_{n})$ be complete graphs with $l(n)$ vertices, where
$l(0)=2$ and $$l(n+1)=\lceil s^{-1}\rceil l(n)^{2},\qquad n\ge0.$$
Enumerate the vertices in $K_{n}$ by
$v_{1}^{(n)},\ldots,v_{l(n)}^{(n)}$ and connect $v_{j}^{(n)}$ with
exactly $(\lceil js^{-1}\rceil-l(n))_{+}$ vertices in $K_{n+1}$ such
that every vertex in $K_{n+1}$ is connected to at most one vertex in
$K_{n}$. This is possible since
\begin{align*}
    \sum_{j=1}^{l(n)}(\lceil js^{-1}\rceil-l(n))_{+}\leq \lceil s^{-1}\rceil
    l(n)^{2}=l(n+1).
\end{align*}
We denote the resulting graph by $G_{s}=(V,E)$. We show that
\begin{align*}
    \limsup_{k\to\infty}\frac{\eta(k)}{k}=s\quad\mbox{and}\quad
    \al_{\infty}=1-\frac{2s}{1+s^{2}}.
\end{align*}
First, we observe that it suffices to show
\begin{align*}
    \limsup_{k\to\infty}\frac{\eta(k)}{k}\leq s\quad\mbox{and}\quad
    \al_{\infty}\leq 1-\frac{2s}{1+s^{2}}.
\end{align*}
Indeed, if $\tilde s=    \limsup_{k\to\infty}\frac{\eta(k)}{k}\leq
s$, then  by Theorem~\ref{t:isoperimetric} we infer
\begin{align*}
 1-\frac{2 s}{1+ s^{2}}
    \leq    1-\frac{2\tilde s}{1+\tilde s^{2}}\leq    \al_{\infty}\leq    1-\frac{2 s}{1+ s^{2}}.
\end{align*}
This implies $\tilde s=s$ and $\alpha_{\infty} =
1- 2s / (1 + s^2)$. For the vertices
$v_{1}^{(n)},\ldots,v_{l(n)}^{(n)}$ in $V_{n}$, we observe for the
vertex degrees in $G_{s}$
\begin{align*}
    \lceil js^{-1}\rceil \vee l(n) -1\leq\deg(v_{j}^{(n)})\leq
    \lceil js^{-1}\rceil \vee l(n)
\end{align*}
since there are $l(n)-1$ neighbors in $V_{n}$, $(\lceil
js^{-1}\rceil-l(n))_{+}$ in $V_{n+1}$ and $0$ or $1$ neighbor in
$V_{n-1}$. \\
Let now $k\ge2$. If $\deg (v_{j}^{(n)})\leq k-1$, we have $\lceil js^{-1}\rceil \vee l(n)\leq k$. This implies
$j \leq ks$. Note further that there is a unique number
$N=N_k \in \mathbb{N}$
such that $l({N_{k}})\le k< l({N_{k}+1})$. It follows that $n\leq N_{k}$.
Hence, for the vertices
$v_{1}^{(N)},\ldots,v^{(N)}_{l(N)}\in V_{N}$, we deduce
\begin{align*}
    \{v\in V\mid \deg(v)\leq k -1\}\subseteq \{v_{j}^{(N)}\mid j\leq
    ks\}\cup\bigcup_{n=1}^{N_{k}-1}V_{n}.
\end{align*}
Thus, using the inequality $2l(n)\leq l(n+1)$ iteratively, yields
\begin{align*}
    \eta(k-1)\leq ks+\sum_{n=1}^{N_{k}-1}l(n)\le ks + 2l(N_k -1).
\end{align*}
We estimate, using $k \geq l(N_{k})$ and $l(N_{k})=\lceil s^{-1}\rceil l(N_{k}-1)^2$,
\begin{align*}
    \frac{\eta(k-1)}{k}
\leq \frac{ks+2l(N_{k}-1)}{k}\leq
s+2\frac{l(N_{k}-1)}{l(N_{k})}=s+2\frac{1}{\lceil
s^{-1}\rceil l(N_{k}-1)},
\end{align*}
and conclude
\begin{align*}
    \limsup_{k\to\infty}    \frac{\eta(k)}{k}\leq s.
\end{align*}
To show $\al_{\infty}\leq1-2s/(1+s^{2})$, we consider $\#\partial
V_{n}/\mathrm{vol}(V_{n})$. We start by estimating
\begin{align*}
    \mathrm{vol}(V_{n})&=\sum_{j=1}^{l(n)}\deg(v_{j}^{(n)})
    \le \sum_{j=1}^{l(n)} ( \lceil js^{-1}\rceil \vee l(n))
    = \int_0^{l(n)} \left( \lceil{ \lceil j \rceil  s^{-1}}\rceil \vee l(n) \right) dj   \\
    &\leq    2s^{-1} l(n) +       s^{-1}   \int_0^{l(n)} \left(  j   \vee  sl(n) \right) dj \\
    &=     2s^{-1} l(n) +   s^{-1} \cdot  \frac 1 2  \left( l(n)^2 + (sl(n))^2 \right)
    %\leq sl(n)(l(n)+1)+s^{-1}\int_{sl(n)}^{l(n)}jdj\\
 %&=
%sl(n)(l(n)+1)+\frac{1}{2s}(l(n)^{2}-(sl(n))^{2})\le
		\le \frac{1+s^{2}}{2s}l(n)(l(n)+4).
\end{align*}
Now, we use the equalities $\deg(V_{n})=2\#E_{n}+\#\partial V_{n}$
and $\#E_{n}=l(n)(l(n)-1)$ to infer
\begin{align*}
    \frac{\#\partial
    V_{n}}{\mathrm{vol}(V_{n})}=1-\frac{2\#E_{n}}{\mathrm{vol}(V_{n})} \leq
    1-\frac{2s}{(1+s^{2})}\frac{(l(n)-1)}{(l(n)+4) }
\end{align*}
which by the discussion given at the beginning implies $\al_{\infty}=1-2s/(1+s^{2})$.
\end{eg}

\section{Exponential volume growth}\label{s:volume}
In this section we prove Theorem~\ref{t:volume}. This is followed by
an example which shows sharpness of the bound.

\begin{proof}[Proof of Theorem~\ref{t:volume}]
Let $s \in (r_{\infty},1)$ and let $n$ be chosen according to
Lemma~\ref{l:basic}. For $v\in V$,  let $k\ge0$ be such that
$\#B_{k}(v)\ge 2n$. We use Lemma~\ref{l:basic} to estimate \begin{align*}
    \mathrm{vol}(B_{k}(v))&=\sum_{w\in
    B_{k}(v)}\deg(w)\ge\sum_{j=0}^{\#B_{k}(v)-1}\deg(v_{j})
    \ge \frac{1}{s}\sum_{j=n}^{\#B_{k}(v)-1}(j+1)\\
    &\ge\frac{1}{s}\int_{\frac 1 2  \# B_{k}(v)}^{\#B_{k}(v)}j
    dj=\frac{3}{8s}(\#B_{k}(v))^{2}.
\end{align*}
Next, we estimate $\#B_{k}$.
By Lemma~\ref{l:basic}, there is a $w \in B_k(v)$ such that $\deg(w) \geq  s^{-1}{\#B_{k}(v)}$.
We obtain $w \notin B_{k-1}(v)$ since else $B_1(w) \subseteq B_k(v)$ which is a contradiction to $s<1$.
Hence,
\begin{align*}
B_1(w) \dot{\cup} B_{k-2}(v) \subseteq B_{k+1} (v)
\end{align*}
and consequently,
\begin{align*}
\# B_{k+1} (v) \geq \# B_1(w) + \# B_{k-2}(v) \geq s^{-1}\#B_{k}(v)  + \# B_{k-2}(v).
\end{align*}
Hence, a lower bound of the growth of $b_k = \# B_k(v)$ is encoded in the eigenvalues of the matrix
\begin{align*}
M_{s}=
\begin{pmatrix}
  0 & 1 & 0 \\
  0 & 0 & 1 \\
  1 & 0 & s^{-1}
 \end{pmatrix}.
\end{align*}
Indeed, the inequality above translates into the componentwise vector inequality $\left(b_{k-1},b_k,b_{k+1}\right)^T \geq M_{s} \left(b_{k-2},b_{k-1},b_{k}\right)^T $.
The characteristic polynomial $p$ of $M_{s}$ is given by $p(z)=z^3-s^{-1}z^2 - 1$ with the largest real root $a_s$.
By the Perron-Frobenius theorem the eigenvector of the irreducible matrix $M_{s}$ to the eigenvalue $a_{s}$ has strictly positive entries. Thus, there is $C_s>0$ such
that for $k$ large enough, we have
$  \#B_{k} (v) \ge C_s {a_s^{k}} $.
By using the estimate $ \mathrm{vol}(B_{k} (v) ) \geq \frac{3}{8s}(\#B_{k} (v) )^{2}$, we obtain
\[
\mathrm{vol}(B_k(v)) \geq D_s a_s^{2k}
\]
for some $D_s>0$ and large enough $k$.
This finishes the proof since $a_s$ is continuous in $s$ for $s>0$ and $a = a_{r_\infty}$.
\end{proof}

Below we give an example for a family of exponentially growing graphs with their values for $r_{\infty}$ being arbitrarily close to zero. These graphs are so called antitrees, see e.g. \cite{KLW,Woj3} and the examples show qualitatively that the above theorem is sharp in the sense of exponential growth.

\begin{eg}
For $\sigma \in \mathbb{N}$ with ${\sigma \geq 2}$, let a graph $G_{\sigma}=(V_{\sigma},E_{\sigma})$ be given with $V=\bigcup_{n\ge0}S_{n}$ such that $\#S_{n}=\sigma^{n}$. Moreover, every vertex in $S_{n}$ is connected to every vertex except to itself in $S_{n-1}\cup S_{n}\cup S_{n+1}$, $n\ge1$.
Hence, for all $n \geq 2$, we have $v \in S_n$ if and only if $\deg v = \sigma^{n-1} + \sigma^n + \sigma^{n+1} -1$.
Let $k \geq \sigma^2 + \sigma $ be an  integer and
 $n \geq 2$ is the unique integer such that
\[
\sigma^{n-2} + \sigma^{n-1} + \sigma^{n} - 1 \leq k < \sigma^{n-1} + \sigma^n + \sigma^{n+1} -1.
\]
Then,
\begin{eqnarray*}
\frac{\eta (k)}{k} \leq
\frac{\eta (\sigma^{n-2} + \sigma^{n-1} + \sigma^{n} - 1)}{\sigma^{n-2} + \sigma^{n-1} + \sigma^{n} - 1} = \frac{\sum_{j=0}^{n-1} \sigma^j}{\sigma^{n-2} + \sigma^{n-1} + \sigma^{n} - 1} <\frac{1}{\sigma-1}.
\end{eqnarray*}
This shows that $r_{\infty,\sigma}:= \limsup_{k\to\infty} \eta (k)/k$
is indeed strictly smaller than one for each $\sigma \geq 2$.
We also derive from the above computation
that
$\lim_{\sigma\to\infty} r_{\infty, \sigma} = 0$.
\end{eg}

There are also examples that show that the precise bound is actually sharp. However, the construction is rather lengthy, so, we refrain from giving the details.

\section{Stochastic incompleteness}\label{s:sc}

This section is devoted to the proof of Theorem~\ref{t:sc} which shows that
rapidly branching graphs with large growth acceleration are stochastically
incomplete.

\begin{proof}[Proof of Theorem~\ref{t:sc}]
By \cite[Theorem~25, Proposition~28]{KL2}
 (cf. also \cite{Woj1,Woj2}) stochastic incompleteness is equivalent
 to existence of  a bounded and positive $\lm$-subharmonic function $u$ for some
$\lm>0$, i.e., $u>0$ satisfies
\begin{align*}
    \sum_{w\sim v}(u(v)-u(w)) +\lm u(v)\leq 0,\qquad v\in V.
\end{align*}
Indeed, it suffices for $u$ to be $\lm$-subharmonic outside of a finite
set, see \cite[Corollary~1.2]{KL1} or \cite[Theorem~4.1]{Hu11}.

We define for    $p\in(0,1)$, $\lm>0$, the function
$u:V\to(0,1)$ by
\begin{align*}
u(v)=1-(\deg(v)+\lm)^{-p},\qquad v\in V.
\end{align*}
We observe
\begin{align*}
\sum_{w\sim v}(u(v)-u(w)) +\lm
u(v)=\lm-(\deg(v)+\lm)^{1-p}+\sum_{w\sim v}(\deg(w)+\lm)^{-p}.
\end{align*}
We proceed by estimating the third term on the right hand side.  Let
$r_{\infty}<s<e^{-1}$ and let $n$ be chosen according to
Lemma~\ref{l:basic}. Below we will
 choose $\lm$ and $p$ such that the
function $u$ becomes $\lm$-subharmonic outside of the set
$K=\bigcup_{k=0}^{n-1}B_{1}(v_{k})$  which is the set of neighbors of $\{v_{0},\ldots,v_{n-1}\}$.
 For
$v\in V\setminus K$,  we find using Lemma~\ref{l:basic}
\begin{align*}
\sum_{w\sim v}(\deg(w)+\lm)^{-p}
&\le \hspace{-.1cm}
    \sum_{k=n}^{n+\deg(v)-1}\hspace{-.4cm}(\deg(v_{k})+\lm)^{-p} \le
    \hspace{-.1cm}\sum_{k=n}^{n+\deg(v)-1}\hspace{-.3cm}\Big(\frac{(k+1)}{s}+\lm \Big)^{-p}\\
    &\le
   {s^{p}}\sum_{k=1}^{\deg(v)}(k+s\lm )^{-p}\leq {s^{p}} \int_{0}^{\deg(v)}(k+s\lm
   )^{-p}dk\\
   &\leq \frac{s^{p}}{1-p}\Big((\deg(v)+\lm)^{1-p}-(s\lm)^{1-p}\Big),
\end{align*}
where we recall  $s<1$ and $p<1$ for the last estimate.
Since we have
$\lim_{p\to 0}(1-p)^{1/p}=e^{-1}$ and $r_{\infty}<s<e^{-1}$,
there is $p$ such that $s\leq(1-p)^{1/p}$. We fix  this choice of $p$ for
what follows and remark
\begin{align*}
    \frac{s^{p}}{1-p}\le1.
\end{align*}
Furthermore,   we let $\lm$ be chosen such that $\lm^{p}\leq
s/(1-p)$. This yields
\begin{align*}
    \lm\le \frac{s}{1-p}\lm^{1-p}= \frac{s^{p}}{1-p}(s\lm)^{1-p}.
\end{align*}
Putting together what we have estimated so far with the equality in
the beginning we find that
\begin{align*}
\lefteqn{\sum_{w\sim v}(u(v)-u(w)) +\lm
u(v)}\\&\le\lm-(\deg(v)+\lm)^{1-p}+
\frac{s^{p}}{1-p}(\deg(v)+\lm)^{1-p}-\frac{s^{p}}{1-p}(s\lm)^{1-p}\\
&\le0.
\end{align*}
According to the discussion in the beginning this finishes the
proof.
\end{proof}

\bibliographystyle{alpha}
\bibliography{literature}

\end{document}